\documentclass{article}

\RequirePackage[OT1]{fontenc}
\RequirePackage[
amsthm,amsmath,natbib]{imsart}

\usepackage{graphicx}
\usepackage{latexsym,amsmath}
\usepackage{amsmath,amsthm,amscd}
\usepackage{amsfonts}
\usepackage[psamsfonts]{amssymb}
\usepackage{enumerate}

\usepackage{url}

\usepackage{tocvsec2}

\usepackage[active]{srcltx}

\startlocaldefs
\numberwithin{equation}{section}

\theoremstyle{plain} 
\newtheorem{theorem}{Theorem}[section]
\newtheorem{corollary}[theorem]{Corollary}

\newtheorem{lemma}[theorem]{Lemma}
\newtheorem{proposition}[theorem]{Proposition}

\theoremstyle{definition} 
\newtheorem{definition}[theorem]{Definition}
\theoremstyle{definition} 

\newtheorem*{ex*}{Example}
\theoremstyle{remark} 

\theoremstyle{remark} 
\newtheorem{remark}[theorem]{Remark}
\newtheorem*{remark*}{Remark}
\numberwithin{equation}{section}
 
\newcommand{\beqa}{\begin{eqnarray}}
\newcommand{\eeqa}{\end{eqnarray}}
 
\newcommand{\bseq}{\begin{subequations}}
 
\newcommand{\eseq}{\end{subequations}}

\newcommand{\XX}{\mathbf{X}}
\newcommand{\xx}{\mathbf{x}}
\newcommand{\yy}{\mathbf{y}}
\newcommand{\zz}{\mathbf{z}}
\newcommand{\s}{\mathbf{s}}
\newcommand{\uu}{\mathbf{u}}
\newcommand{\vv}{\mathbf{v}}
\newcommand{\ww}{\mathbf{w}}
\renewcommand{\aa}{\mathbf{a}}
\newcommand{\bb}{\mathbf{b}}
\newcommand{\ee}{\mathbf{e}}

\newcommand{\0}{\mathbf{0}}
\newcommand{\one}{\mathbf{1}}

\newcommand{\dd}{\partial}

\newcommand{\fb}{\fbox{\rule{0pt}{5pt}\hspace*{-5pt}\rule{0pt}{5pt}}}
\newcommand{\no}[1]{\fb\,#1\,\fb\,}

\renewcommand{\dd}{{\,\operatorname{d}}}

\newcommand{\sign}{\operatorname{sign}}
\newcommand{\mes}{\operatorname{mes}}

\newcommand{\are}{\operatorname{ARE}}

\newcommand{\ssim}{\text{\raisebox{3pt}[0pt][0pt]{$\sim$}}}
\newcommand{\lesim}{\underset{\ssim}{<}}
\newcommand{\gesim}{\underset{\ssim}{>}}
\newcommand{\lsim}{\mathrel{\text{\raisebox{2pt}{$\lesim$}}}}
\newcommand{\gsim}{\mathrel{\text{\raisebox{2pt}{$\gesim$}}}}

\newcommand{\al}{\alpha}
\newcommand{\g}{\gamma}

\newcommand{\si}{\sigma}
\newcommand{\Si}{\Sigma}

\newcommand{\la}{\lambda}
\newcommand{\de}{\delta}

\newcommand{\be}{\beta}

\newcommand{\vpi}{\varphi}
\renewcommand{\th}{\theta}
\newcommand{\thb}{{\boldsymbol{\theta}}}

\newcommand{\ii}[1]{\operatorname{I}\{#1\}} 
\newcommand{\bigii}[1]{\operatorname{I}\big\{#1\big\}}

\newcommand{\PP}{\operatorname{\mathsf{P}}} 
\newcommand{\E}{\operatorname{\mathsf{E}}}

\newcommand{\Cov}{\operatorname{\mathsf{Cov}}}

\newcommand{\R}{\mathbb{R}}

\newcommand{\ZZ}{\mathbf{Z}}

\newcommand{\A}{\mathcal{A}}

\newcommand{\vp}{\varepsilon}

\renewcommand{\le}{\leqslant}
\renewcommand{\ge}{\geqslant}
\renewcommand{\leq}{\leqslant}
\renewcommand{\geq}{\geqslant}
\endlocaldefs

\begin{document}

\begin{frontmatter}

\title{Schur${}^2$-concavity properties of Gaussian measures, with applications to hypotheses testing}
\runtitle{Schur${}^2$-concavity and testing
}
\date{\today}

\begin{aug}
\author{\fnms{Iosif} \snm{Pinelis}\ead[label=e1]{ipinelis@mtu.edu}}
\runauthor{Iosif Pinelis}


\address{Department of Mathematical Sciences\\
Michigan Technological University\\
Houghton, Michigan 49931, USA\\
E-mail: \printead[ipinelis@mtu.edu]{e1}
}
\end{aug}

\begin{abstract}
The main results imply that the probability $\PP(\ZZ\in A+\thb)$ is Schur-concave/Schur-convex in $(\th_1^2,\dots,\th_k^2)$ provided that the indicator function of a set $A$ in $\R^k$ is so, respectively; here, $\thb=(\th_1,\dots,\th_k)\in\R^k$ and $\ZZ$ is a standard normal random vector in $\R^k$.  
Moreover, it is shown that the Schur-concavity/Schur-convexity is strict unless the set $A$ is equivalent to a spherically symmetric set.  
Applications to testing hypotheses on multivariate means are given. 
\end{abstract}

\begin{keyword}[class=AMS]
\kwd[Primary ]{60E15}
\kwd{62H15}
\kwd{62F05}
\kwd{62G20}
\kwd[; secondary ]{60D05}
\kwd{62E15}
\kwd{62E20}
\kwd{60D05}
\kwd{51F15}
\kwd{53C65}
\end{keyword}
\begin{keyword}
\kwd{probability inequalities}
\kwd{geometric probability}
\kwd{Gaussian measures}
\kwd{multivariate normal distribution}
\kwd{mixtures}
\kwd{majorization}
\kwd{stochastic ordering}
\kwd{Schur convexity}
\kwd{hypothesis testing}
\kwd{asymptotic properties of tests}
\kwd{asymptotic relative efficiency}
\kwd{$p$-mean tests}
\kwd{multivariate means}
\kwd{reflection groups}
\end{keyword}

%
%

\end{frontmatter}

\settocdepth{chapter}

\tableofcontents 

\settocdepth{subsubsection}

\theoremstyle{plain} 
\numberwithin{equation}{section}

\section{Introduction}\label{intro} 
Consider the classical problem of testing for a multivariate mean (say $\thb\in\R^k$) given a large sample of (say) independent identically distributed observations. Suppose that, under the null hypothesis $H_0\colon\thb=\0$, the covariance matrix of the underlying distribution is known and nonsingular, and thus without loss of generality may be assumed to coincide with the identity matrix.   
Then, by the central limit theorem, the power of the approximate likelihood ratio test (LRT) on an alternative mean vector $\thb$ is $\PP(\|\ZZ-\thb\|>c)$, where $\|\cdot\|$ is the Euclidean norm, $\ZZ$ is a standard normal random vector in $\R^k$, and $c$ is the critical value. It is of interest to compare the asymptotic efficiency of the approximate LRT with that of other tests. For instance, one can consider tests based -- instead of the Euclidean norm $\|\cdot\|$ -- on the general $\ell_p$-norm $\|\cdot\|_p$ on $\R^k$, according to the definition \label{ell_p} $\|\xx\|_p:=(|x_1|^p+\cdots+|x_k|^p)^{1/p}$ for $\xx=(x_1,\dots,x_k)\in\R^k$; let us refer to such tests as the  $p$-tests, for brevity;  
then the approximate LRT is a special case of a $p$-test, with $p=2$. 
The asymptotic relative efficiency, say $\are_{p,2}$, of the $p$-test relative to the $2$-test will depend on the direction of the alternative mean vector $\thb$, so that one can write $\are_{p,2}=\are_{p,2,\uu}$, where $\uu$ is (say) the $\|\cdot\|_2$-unit vector in the direction of $\thb$. 
Therefore, it turns out quite useful to have monotonicity properties of the approximate power $\PP(\|\ZZ-\thb\|_p>c)$ for a varying direction of the vector $\thb$. 

The main results of the present paper -- Theorem~\ref{th:schur} and Corollary~\ref{cor:schur} -- imply that $\PP(\ZZ\in A+\thb)$ is Schur-concave or Schur-convex in $(\th_1^2,\dots,\th_k^2)$ provided that the indicator function of a set $A$ in $\R^k$ is so, respectively; of course, the $\th_j$'s denote the coordinates of $\thb$; moreover, we show that, for such a set $A$, the probability $\PP(\XX\in A+\thb)$ will be \emph{strictly} Schur-concave/Schur-convex in $(\th_1^2,\dots,\th_k^2)$ unless the set $A$ is equivalent (up to a set of Lebesgue measure $0$) to a spherically symmetric set.  
\big(For brevity, we shall refer to the Schur-concavity/Schur-convexity in $(\th_1^2,\dots,\th_k^2)$ as the Schur${}^2$-concavity/Schur${}^2$-convexity in $\thb$.\big)  

Note that the power $\PP(\|\ZZ-\thb\|_p>c)$ can be written as $\PP(\ZZ\in A_{p,c}+\thb)$, where $A_{p,c}:=\{\xx\in\R^k\colon\|\xx\|_p>c\}$; moreover, one can see that the indicator of $A_{p,c}$ is Schur${}^2$-concave for $p\ge2$ and Schur${}^2$-convex for $p\le2$. 
Based on these facts, it can be shown (see Corollary~\ref{cor:are schur2} in this paper) 
that $\are_{p,2,\uu}$ is Schur${}^2$-convex in $\uu$ for each $p\ge2$ and Schur${}^2$-concave in $\uu$ for each $p\le2$. Informally, this means that, for each $p\ge2$, $\are_{p,2,\uu}$ is greater when the direction of the alternative vector $\thb$ is closer to that of one of the 
$2k$ ``coordinate'' vectors $(0,\dots,0,\pm1,0,\dots,0)\in\R^k$, whereas, for each $p\le2$, $\are_{p,2,\uu}$ is greater when the direction of the alternative vector $\thb$ is closer to that of one of the $2^k$ ``diagonal'' vectors $(\pm1,\dots,\pm1)\in\R^k$. 

An application of these results is as follows. It is not hard to show (see the proof of Corollary~\ref{cor:are<1} in this paper) that, for $p\le2$ and $\uu$ in the direction of the ``diagonal'' vector $\one:=(1,\dots,1)\in\R^k$, the upper limit of $\are_{p,2,\uu}$ is no greater than $1$ as both types of error probabilities (say $\al$ and $1-\be$) tend to $0$; thus, by the Schur${}^2$-concavity of $\are_{p,2,\uu}$ in $\uu$, the same holds for any direction $\uu$. Similarly, for $p\ge2$ and $\uu$ in the direction of the ``coordinate'' vector $\ee_1:=(1,0,\dots,0)\in\R^k$, the upper limit of $\are_{p,2,\uu}$ is no greater than $1$ as $\al,1-\be\to0$; thus, the same holds for any direction $\uu$. So, the approximate LRT asymptotically outperforms all the competing $p$-tests as $\al,1-\be\to0$. 

The situation turns out drastically different when the error probabilities $\al$ and $1-\be$ are fixed, whereas the dimension $k$ grows to $\infty$. 
Then, for each $p>2$, the values of $\are_{p,2,\uu}$ can be arbitrarily large if the direction of $\uu$ stays close enough to the ``coordinate'' one, whereas $\are_{p,2,\uu}$ is well bounded away from $0$ even for the worst possible, ``diagonal'' directions $\uu$ \cite{d_to_infty}. 
For instance, for $p=3$ the possible values of $\are_{p,2,\uu}$ range from $\approx0.96$ to $\infty$ as the direction of $\uu$ varies from the ``diagonal'' to ``coordinate'' ones.   
This suggests that (say) the $3$-test should generally be preferred to the LRT for large dimensions $k$ -- especially, when the direction of the alternative vector $\thb$ is unknown or, even more so, when the direction of $\thb$ is known to be far from the ``diagonal'' one. One can say that the $3$-test is much more robust than the LRT with respect to a few large coordinates of an alternative mean vector $\thb$, while the asymptotic efficiency of the $3$-test relative to the LRT never falls below $\approx96\%$.

The main results of this paper, mentioned above, are somewhat similar in form to the classic result by  
Marshall and Olkin \cite[Th.\ 2.1]{marsh-olkin74}: If a random vector $\XX$ in $\R^k$ has a Schur-concave density and $A$ is a Schur-convex set in $\R^k$ (that is, the indicator of $A$ is Schur-concave), then $\PP(\XX\in A+\xx)$ is Schur-concave in $\xx\in\R^k$. 
The proof in \cite{marsh-olkin74} was obtained by a clever reduction to another classic result, due to Anderson \cite{anderson55}, 
which might at the first glance seem to be of rather different form: 

Let a random vector $\XX$ in $\R^k$ have a nonnegative even density $f$, which is unimodal in the sense that the sets $\{\xx\in\R^k\colon f(\xx)\ge u\}$ are convex for all real $u$. 
Let $A$ be a convex set in $\R^k$, symmetric about the
origin. 
Then $\PP(\XX\in A+\la\xx)$ 
is non-increasing in $\la\ge0$ for any given $\xx\in\R^k$. 
In turn, this result of \cite{anderson55} is proved by noticing that without loss of generality $f$ may be assumed to be an indicator function and then using the Brunn-Minkowski inequality for the volumes of convex bodies.   

Another result, sounding similarly to Marshall and Olkin's, 
is due to Jogdeo \cite[Th.\ 2.2]{jog77}: Let a random vector $\XX$ in $\R^k$ have an axially unimodal density $f$ and let $A$ be an axially convex set. Then $\PP(\XX\in A+\xx)$ is axially unimodal in $\xx\in\R^k$. 	
The proof in \cite{jog77} is also obtained based on the observation that the density $f$ may without loss of generality be assumed to be an indicator function. 

The particular case of our Corollary~\ref{cor:schur} when the set $A$ is a $k$-dimensional cube of the form $[-a,a]^k$ was obtained by Hall, Kanter, and Perlman (HKP) \cite{HKP}. The HKP result was complemented by the Mathew and Nordstr\"om \cite{mathew} study of the probability content of a rotated ellipse; our result appears to be independent from that of \cite{mathew} in the sense that neither of these two results helps one to obtain the other. 

In contrast with mentioned resuts of \cite{marsh-olkin74,jog77}, our Theorem~\ref{th:schur} and Corollary~\ref{cor:schur} cannot be reduced to  densities which are the indicator functions of ``nice'' sets (see Proposition~\ref{prop:counterex}). No such reduction is possible even for the HKP result, as pointed out in the penultimate paragraph on page~810 in \cite{HKP}. However, what facilitates the proof in \cite{HKP} is that, in the case when $A$ is the $k$-cube $[-a,a]^k$, the probability $\PP(\ZZ\in A+\thb)$ is the product of ``one-dimensional'' probabilities $\PP(Z_j\in[-a,a]+\th_j)$ over $j=1,\dots,k$. 
On the other hand, we have to deal with the fact that, for a general set $A$ with a Schur${}^2$-concave/Schur${}^2$-convex indicator, there is no such nice factorization property. 
One may also note that the HKP method is purely analytic (based on log-convexity properties for Laplace and related transforms, cf.\ \cite{pin99}), whereas the key idea in this paper is based on group symmetry. 


For general treatises of related problems, we refer to the well-known monographs by Marshall and Olkin \cite{marsh-ol}, Dharmadhikari and Joag-Dev \cite{dharm-joag88}, and Tong \cite{tong}. 

The paper is organized as follows. The results are stated and discussed in Section~\ref{results}: general results on the Schur${}^2$-convexity/concavity of Gaussian measures in Subsection~\ref{general} and applications to testing for multivariate means in Subsection~\ref{means}. The proofs are deferred to Section~\ref{proofs}.

\section{Results}\label{results}
\subsection{General results on the Schur${}^2$-concavity of Gaussian measures}\label{general}
Recall that, for any $E\subseteq\R^k$, a function $f\colon E\rightarrow[0,\infty]$ is referred to as Schur-concave if it reverses the
Schur majorization: for any $\mathbf{a}$ and $\mathbf{b}$ in $E$ such that $%
\mathbf{a\succeq b}$, one has $f\left( \mathbf{a}\right) \leq%
f\left( \mathbf{b}\right) $. Recall also the definition of the
Schur majorizarion: for $\mathbf{a:=}\left( a_{1},\ldots,a_{k}\right) $ and $%
\mathbf{b:=}\left( b_{1},\ldots,b_{k}\right) $ in $\R^k$, $\mathbf{%
a\succeq b}$ (or, equivalently, 
$\mathbf{%
b\preceq a}$)
means that $a_{1}+\cdots+a_{k}=b_{1}+\cdots+b_{k}$ and $a_{%
\left[ 1\right] }+\cdots+a_{\left[ j\right] }\geq b_{\left[ 1\right]
}+\cdots+b_{\left[ j\right] }$ for all $j\in\left\{ 1,\ldots,k\right\} $,
where $a_{\left[ 1\right] }\geq\cdots\geq a_{\left[ k\right] }$ are the
ordered numbers $a_{1},\ldots,a_{k}$, from the largest to the smallest. 
As
usual, let $\mathbf{a}_{\downarrow}:=\left( a_{\left[ 1\right] },\ldots,a_{%
\left[ k\right] }\right) $. If $\mathbf{a}\succeq\mathbf{b}$ and $\mathbf{a}%
_{\downarrow}\neq\mathbf{b}_{\downarrow}$, let us write $\mathbf{a}\succ%
\mathbf{b}$. 
If for all $\mathbf{a}$ and $\mathbf{b}$ in $E$ such that $%
\mathbf{a\succ b}$, one has $f\left( \mathbf{a}\right) <%
f\left( \mathbf{b}\right) $ or $f\left( \mathbf{a}\right) =
f\left( \mathbf{b}\right) =\infty$, let us say that the function $f\colon E\rightarrow[0,\infty]$ is strictly Schur-concave.

\begin{definition}\label{def:2conc}
Let us say that a Lebesgue-measurable function $f\colon\R^k\rightarrow[0,\infty]$ is (strictly) Schur${}^2$-concave if $f(\xx)=f(x_1,\dots,x_k)$ is (strictly) Schur-concave in $(x_1^2,\dots,x_k^2)$, that is, if there exists a (strictly) Schur-concave function $g\colon[0,\infty)^k\to[0,\infty]$ such that $f(x_1,\dots,x_k)=g(x_1^2,\dots,x_k^2)$ for all $(x_1,\dots,x_k)\in\R^k$. 

For brevity, let $\xx^2:=(x_1^2,\dots,x_k^2)$ for any $\xx=(x_1,\dots,x_k)\in\R^k$. 

Let us then say that a set $A\subseteq\R^k$ is Schur${}^2$-convex if its indicator function is  Schur${}^2$-concave; that is, if the conditions $\xx\in A$, $\yy\in\R^k$, and $\yy^2\preceq\xx^2$ imply $\yy\in A$. Say that $A\subseteq\R^k$ is Schur${}^2$-concave if the complement $\R^k\setminus A$ is Schur${}^2$-convex; that is, if the conditions $\xx\in A$, $\yy\in\R^k$, and $\yy^2\succeq\xx^2$ imply $\yy\in A$. 
\end{definition}

\begin{remark}\label{rem:symm}
Any Schur${}^2$-concave function on $\R^k$ and, hence, any Schur${}^2$-convex set in $\R^k$ are invariant with respect to the linear group, say $G_k$, generated by the reflections $(x_1,\dots,x_k)\mapsto(x_1,\dots,x_{i-1},-x_i,x_{i+1},\dots,x_k)$ and 
$(x_1,\dots,x_k)\mapsto(x_1,\dots,x_{i-1},x_j,x_{i+1},\dots,x_{j-1},x_i,x_{j+1},\dots,x_k)$ about the coordinate hyperplanes $\{(x_1,\dots,x_k)\in\R^k\colon x_i=0\}$ ($i=1,\dots,k$) and 
the diagonal hyperplanes $\{(x_1,\dots,x_k)\in\R^k\colon x_i=x_j\}$ ($i,j=1,\dots,k$, $i<j$).  
One may recognize $G_k$ is the group of symmetries of the hypercube (or, equivalently, of the cross polytope) in $\R^k$; its order is $2^k k!$, as each member of the group is the product of a permutation of the $k$ coordinates  and up to $k$ sign switches; see e.g.\ \cite[pages~130--133]{coxeter}. 
In particular, the group $G_2$ 
is known as a dihedral group and denoted as $D_4$ or $D_8$; it consists of $8$ linear transformations, mapping each point $(u,v)\in\R^2$ to $(u,v)$, $(-u,v)$, $(u,-v)$, $(-u,-v)$, $(v,u)$, $(-v,u)$, $(v,-u)$, and $(-v,-u)$, respectively. 
\end{remark}

%

A trivial example of functions that are Schur${}^2$-concave but not strictly Schur${}^2$-concave is that of the \emph{spherically symmetric} functions, that is, the measurable functions that are constant on each sphere $S_r:=\{\xx\in\R^k\colon\|\xx\|=r\}$, for each $r\in(0,\infty)$.  
(As usual, $\|\xx\|:=(x_1^2+\dots+x_k^2)^{1/2}$, the Euclidean norm of a vector $\xx\in\R^k$. Also, we shall denote vectors in $\R^k$ by boldfaced letters and their coordinates, by the corresponding italicized, not boldfaced letters with indices.) Similarly one can consider (measurable) spherically symmetric sets --- whose indicators are spherically symmetric; obviously, any such set is the union of some family of spheres $S_r$.  
Other examples of Schur${}^2$-convex/concave functions/sets are given in Section~\ref{means}; see Propositions~\ref{prop:means}, \ref{prop:pq-means}, \ref{prop:hat,check B}, Corollary~\ref{cor:means1}, and  Remark~\ref{rem:rearr}. 
Further examples can be easily obtained by observing that the classes of all Schur${}^2$-convex/concave functions are each invariant with respect to taking linear combinations of such functions with nonnegative coefficients, taking (say, finite) pointwise suprema/infima, taking limits, as well as applying increasing transformations to the ranges of such functions. 

Similarly, one can define Schur-convex sets as the ones whose indicator is  Schur-concave. 
Neither the class of all Schur${}^2$-convex sets nor the class of all Schur-convex sets is contained in the other. Indeed, the complement to $\R^k$ of the unit Euclidean ball in $\R^k$ centered at the origin is Schur${}^2$-convex but not Schur-convex; on the other hand, the mentioned cross polytope (that is, the unit ball in $\R^k$ in the $\ell^1$-norm) is Schur-convex but not Schur${}^2$-convex. This remark also implies the mutual lack of containment between the class of all Schur${}^2$-convex (respectively, concave) functions and, on the other hand, the class of all Schur-convex (respectively, concave) functions.  

Let $\ZZ:=(Z_1,\dots,Z_k)$ stand for a standard normal random vector in $\R^k$, with zero mean and identity covariance matrix. 

As usual, let us say that two (Lebesgue) measurable functions on $\R^k$ are equivalent (to each other) if they are equal almost everywhere (a.e.), that is, if they 
differ only on a set of zero Lebesgue measure. Let us say that two measurable subsets of $\R^k$ are equivalent if their indicators are so. The basic result in this paper is 

\begin{theorem}\label{th:schur}
Suppose that a function $f\colon\R^k\to[0,\infty]$ is Schur${}^2$-concave. Take any $\si\in(0,\infty)$. 
\begin{enumerate}[(I)]
	\item  Then the function 
\begin{equation*}
\R^k\ni\xx\mapsto f^{(\si)}(\xx):=\E f(\si\ZZ+\xx) 	
\end{equation*}
is Schur${}^2$-concave as well. 
\item
Moreover, $f^{(\si)}$ is \emph{strictly} Schur${}^2$-concave unless $f$ 
is equivalent to a spherically symmetric function (in which latter case $f^{(\si)}$ is spherically symmetric, too). 
\end{enumerate}
\end{theorem}

This immediately yields

\begin{corollary}\label{cor:schur}
Suppose that a set $A\subseteq\R^k$ is Schur${}^2$-convex.
\begin{enumerate}[(I)]
	\item  Then the function 
\begin{equation*}
\R^k\ni\xx\mapsto P(\xx):=\PP(\ZZ\in A+\xx) 	
\end{equation*}
is Schur${}^2$-concave. 
\item
Moreover, the function $P$ is \emph{strictly} Schur${}^2$-concave unless $A$ 
is equivalent to a spherically symmetric set. 
\end{enumerate}
\end{corollary}

On the other hand, one has

\begin{proposition}\label{prop:counterex}
For each integer $k\ge2$ there exist a random vector $\XX$ in $\R^k$ uniformly distributed in a Euclidean ball in $\R^k$ centered at the origin 
and a Schur${}^2$-convex set $A\subseteq\R^k$ such that the function 
$
\R^k\ni\xx\mapsto\PP(\XX\in A+\xx) 	
$
is \emph{not} Schur${}^2$-concave.
\end{proposition}
\noindent
A similar counterexample was given in \cite{HKP} for $k=2$. 
Note that the density of the distribution of the random vector $\XX$ in Proposition~\ref{prop:counterex} is spherically symmetric, Schur-concave, Schur${}^2$-concave, and hence unimodal -- in the sense that it is a nonincreasing function of the distance from the point to the origin.

Thus, the condition in Theorem~\ref{th:schur} and Corollary~\ref{cor:schur}
that the random vector $\ZZ$ be normally distributed is essential.   
However, it is clear that Theorem~\ref{th:schur} and Corollary~\ref{cor:schur} will hold for any 
random vector $\XX$ whose distribution is any mixture of spherically symmetric normal distributions in $\R^k$; cf.\ \cite[Theorem~3.1]{HKP}; such an $\XX$ equals in distribution the random vector of the form $\Lambda\ZZ$, where $\Lambda$ is a positive random variable independent of $\ZZ$; at that, if $1/\Lambda^2$ has a $\chi^2$ distribution, then $\XX$ has a  multivariate $t$ distribution; cf.\ Marshall and Olkin \cite[\S4]{marsh-olkin74}; see also Strawderman \cite{straw} for a discussion of mixtures of spherically symmetric normal distributions. 

\subsection{Application to $p$-mean tests for multivariate means}\label{means}
For any $p\in(-\infty,\infty)\setminus\{0\}$, define the $p$-mean of any vector $\xx=(x_1,\dots,x_k)\in\R^k$ by the formula 
\begin{equation}\label{eq:_p}
	\no\xx_p
	:=\Big(\frac1k\,\sum\limits_{j=1}^k|x_j|^p\Big)^{1/p}; 
\end{equation} 
at that, if $p\in(-\infty,0)$, use the continuity conventions $0^p:=\infty$ and $\infty^{1/p}:=0$, so that $\no\xx_p=0$ if $p\in(-\infty,0)$ and at least one of the $x_j$'s is $0$. 
As usual, extend definition \eqref{eq:_p} of the $p$-mean by continuity to all $p\in[-\infty,\infty]$, so that 
\begin{equation*}
	\no\xx_{-\infty}=\min_{j=1}^k|x_j|,\quad 
	\no\xx_0=\prod_{j=1}^k|x_j|^{1/k},\quad 
	\no\xx_\infty=\max_{j=1}^k|x_j|.  
\end{equation*}
For $p\in[1,\infty]$, the function $\no\cdot_p$ is a norm, which differs from the more usual $p$-norm by the factor $\frac1{k^{1/p}}$; this factor is needed in order that $\no\xx_p\to\no\xx_0$ as $p\to0$. 
Note that 
\begin{equation}\label{eq:one<u<e}
	\one=\one^2\preceq\uu^2\preceq(\sqrt k\,\ee_1)^2=k\,\ee_1, 
\end{equation}
for any $\no\cdot_2$-unit vector $\uu\in\R^k$, where 
\begin{equation*}
	\one:=(1,\dots,1)\in\R^k\quad\text{and}\quad\ee_1:=(1,0,\dots,0)\in\R^k. 
\end{equation*}
Let 
\begin{equation*}
	B_p(\vp):=\{\xx\in\R^k\colon\no\xx_p\le\vp\}, 
\end{equation*}
the ``ball'' of ``radius'' $\vp$ with respect to the ``norm'' $\no\cdot_p$. 

\begin{proposition}\label{prop:means}
The $p$-mean $\no\cdot_p$ is Schur${}^2$-concave for $p\in[-\infty,2]$ and Schur${}^2$-convex for $p\in[2,\infty]$. 
\end{proposition}

This proposition, which is a special case of Proposition~\ref{prop:pq-means} below, immediately implies

\begin{corollary}\label{cor:means1}
The ``ball'' $B_p(\vp)$ is Schur${}^2$-concave for $p\in[-\infty,2]$ and Schur${}^2$-convex for $p\in[2,\infty]$. 
\end{corollary}

So, by Corollary~\ref{cor:schur}, one has 

\begin{corollary}\label{cor:means2}
The function 
\begin{equation*}
\R^k\ni\xx\mapsto \PP(\no{\ZZ+\xx}_p\le\vp) 	
\end{equation*}
is Schur${}^2$-convex for $p\in[-\infty,2]$ and Schur${}^2$-concave for $p\in[2,\infty]$. 
\end{corollary}


\begin{remark}\label{rem:rearr}
More generally, Proposition~\ref{prop:means} and Corollaries~\ref{cor:means1} and \ref{cor:means2} hold if the $p$-mean $\no\cdot_p$ is replaced by $\no\cdot_{\ell,\uparrow;p}$ if $p\in[-\infty,2]$ and by $\no\cdot_{\ell,\downarrow;p}$ if $p\in[2,\infty]$, where $\ell$ is any integer in the set $\{1,\dots,k\}$, $\no\uu_{\ell,\uparrow;p}$ is the $p$-mean of the vector in $\R^\ell$ whose coordinates are the $\ell$ smallest absolute values of the coordinates of a vector $\uu\in\R^k$, and $\no\uu_{\ell,\downarrow;p}$ is defined similarly by taking the $\ell$ largest absolute values.  
In particular, $\no\cdot_{k,\uparrow;p}=\no\cdot_{k,\downarrow;p}=\no\cdot_p$, $\no\cdot_{1,\uparrow;p}=\no\cdot_{-\infty}$, and $\no\cdot_{1,\downarrow;p}=\no\cdot_\infty$ for all $p\in[-\infty,\infty]$.  
For $p\ge1$, the means $\no\cdot_{\ell,\downarrow;p}$ are used in geometric theory of Banach spaces; see e.g.\ \cite{gordon02}.    
\end{remark}

Another generalization of the $p$-mean is what we shall refer to as the $(p,q)$-mean, defined as follows: for any $\xx\in\R^k$ with all coordinates $x_j$ nonzero,  
\begin{equation*}
	\no\xx_{p,q}
	:=\Big(\sum\limits_{j=1}^k|x_j|^p\Big/\sum\limits_{j=1}^k|x_j|^q\Big)^{\frac1{p-q}} 
\end{equation*}  
if $p$ and $q$ are any real numbers such that $p\ne q$. This definition can be extended by continuity to the case when some of the $x_j$'s are zero. Moreover, also by continuity, one has 
\begin{equation*}
	\no\xx_{p,p}
	:=\prod\limits_{j=1}^k|x_j|^{|x_j|^p\big/\sum_{i=1}^k|x_i|^p}, 
\end{equation*} 
with the conventions $0^0:=1$ and $\no\0_{p,p}:=0$. 
Also, one can extend the definition of $\no\cdot_{p,q}$ by continuity to include the cases when $p$ or $q$ equal $\infty$ or $-\infty$, so that for each $p\in\R$ the $(p,q)$-mean $\no\cdot_{p,q}$ is nondecreasing in $q\in[-\infty,\infty]$ from $\no\cdot_{-\infty}$ to $\no\cdot_\infty$. 
Moreover, $\no\cdot_{p,0}=\no\cdot_p$ and $\no\cdot_{p,q}=\no\cdot_{q,p}$, so that without loss of generality $p\ge q$. 

Consider the corresponding ``balls'' in $\R^k$: 
\begin{equation*}
	B_{p,q}(\vp):=\{\xx\in\R^k\colon\no{\xx}_{p,q}\le\vp\}, \quad\text{so that}\quad 
	B_{p,0}(\vp)=B_p(\vp).  
\end{equation*}
Also, for any $a\ge0$, let us introduce the sets 
\begin{equation*}
	\hat B_p(a,\vp):=\bigcup_{g\in G_k}\big(ag\one+B_p(\vp)\big)\quad\text{and}\quad 
	\check B_p(a,\vp):=\bigcup_{g\in G_k}\big(ag\ee_1+B_p(\vp)\big), 
\end{equation*}
where $G_k$ is the group of transformations defined in Remark~\ref{rem:symm}. 

A few of these sets are shown in Figure~\ref{fig1}, for dimension $k=2$. 
\begin{figure}[htbp]
	\centering
		\includegraphics[width=1.00\textwidth]{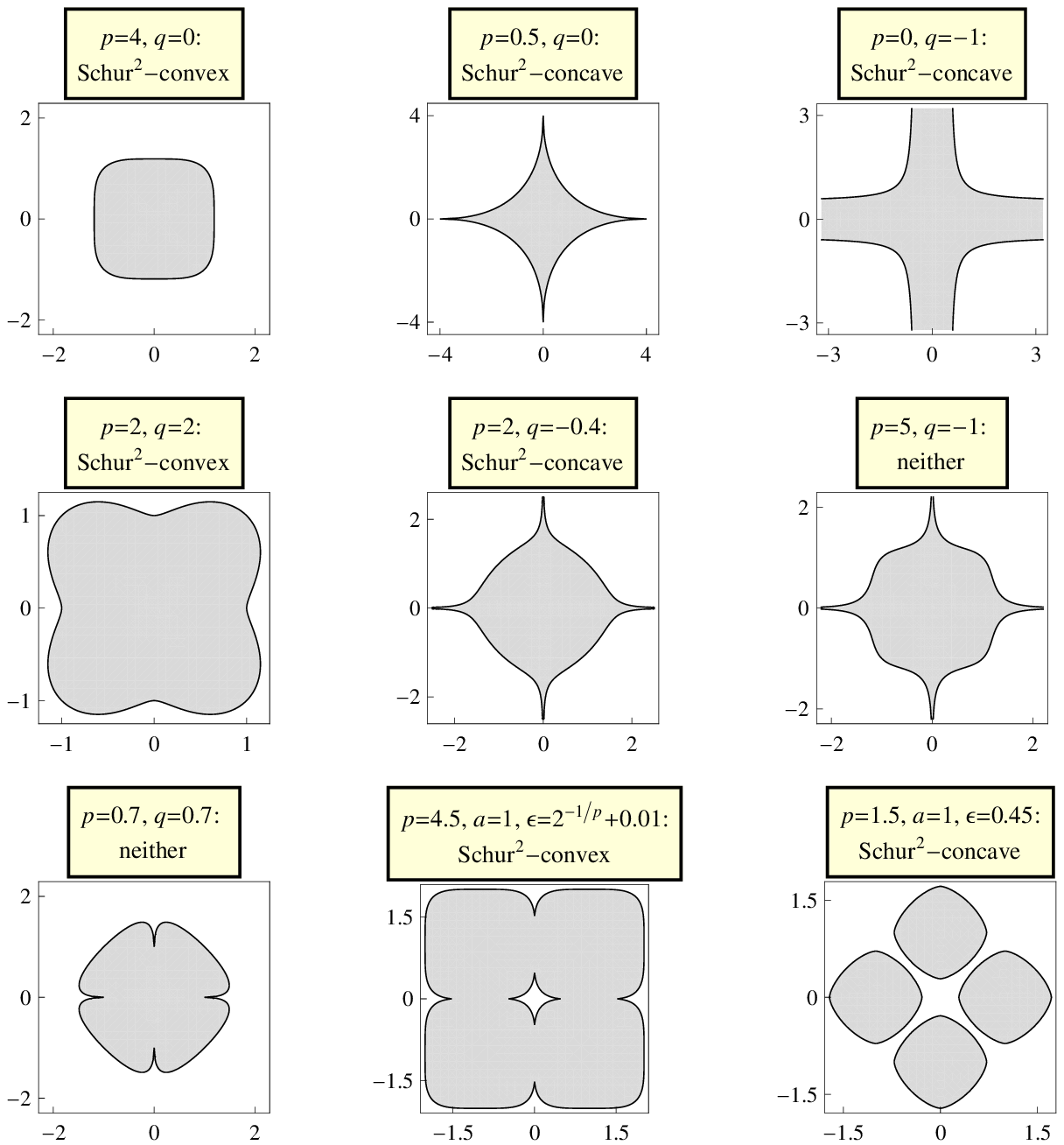}
	\caption{Sets $B_{p,q}(\vp)$, $\hat B_p(a,\vp)$, and $\check B_p(a,\vp)$.}
	\label{fig1}
\end{figure}
The first 7 pictures (the top two rows and the leftmost of the bottom row) show the unit ``balls'' $B_{p,q}(1)$ for 7 selected pairs $(p,q)$; the last two pictures show the sets $\hat B_{4.5}(1,2^{-1/4.5}+0.01)$ and $\check B_{1.5}(1,\frac12-0.05)$, respectively. By Propositions~\ref{prop:pq-means} and \ref{prop:hat,check B}, 7 
of these 9 sets are either Schur${}^2$-convex or Schur${}^2$-concave, while the other two -- $B_{5,-1}(1)$ and $B_{0.7,0.7}(1)$ -- are neither. 
If $p>q$ and $q<0$, the set $B_{p,q}(\vp)$ contains the entire coordinate axes and therefore is unbounded; so, Figure~\ref{fig1} shows only parts of the ``balls'' $B_{p,q}(1)$ with $(p,q)=(0,-1),\;(2,-0.4),\;(5,-1)$. 

\begin{proposition}\label{prop:pq-means}
The $(p,q)$-mean $\no\cdot_{p,q}$ is Schur${}^2$-concave if $q\le0\le p\le2$ and Schur${}^2$-convex if $0\le q\le2\le p$. In particular, $\no\cdot_{2,2}$ is Schur${}^2$-convex. 
Thus, Corollary~\ref{cor:schur} holds with $B_{p,q}(\vp)$ in place of $A$ if $0\le q\le2\le p$, and it holds with the complement of $B_{p,q}(\vp)$ (to $\R^k$) in place of $A$ if $q\le0\le p\le2$. 
\end{proposition}



\begin{proposition}\label{prop:hat,check B}
The sets $\hat B_p(a,\vp)$ are Schur${}^2$-convex for $p\ge2$, and the sets $\check B_p(a,\vp)$ are Schur${}^2$-concave for $p\in[1,2]$. 
Thus, Corollary~\ref{cor:schur} holds with $\hat B_p(a,\vp)$ in place of $A$ if $p\ge2$, and it holds with the complement of $\check B_p(a,\vp)$ in place of $A$ if $p\in[1,2]$.
\end{proposition}

%

Proposition~\ref{prop:pq-means} is illustrated in Figure~\ref{fig2}, for dimension $k=2$: \\ 
$\PP\big(\ZZ\in B_{2,-0.4}(1)+11R^{\pi/5}\ee_1\big)<\PP\big(\ZZ\in B_{2,-0.4}(1)+11R^{\pi/20}\ee_1\big)$, where $R^t$ denotes the rotation through angle $t$. 
In fact, these probabilities are very different from each other: the first is ${}\approx1.5\times10^{-14}$, while the second is $\approx1.4\times10^{-6}$ (because the set $B_{2,-0.4}(1)$ is rather complicated, it takes about 5 min to compute either one of these two probabilities in Mathematica on a standard Core 2 Duo laptop, even after some preparations). However, the probabilities $\PP\big(\ZZ\in B_{2,-0.4}(1)+R^{\pi/5}\ee_1\big)$ and $\PP\big(\ZZ\in B_{2,-0.4}(1)+R^{\pi/20}\ee_1\big)$ (without the factor $11$) are much closer to each other: $\approx0.5250$ and $\approx0.5268$, respectively. 
\begin{figure}[htbp]
	\centering
		\includegraphics[width=0.60\textwidth]{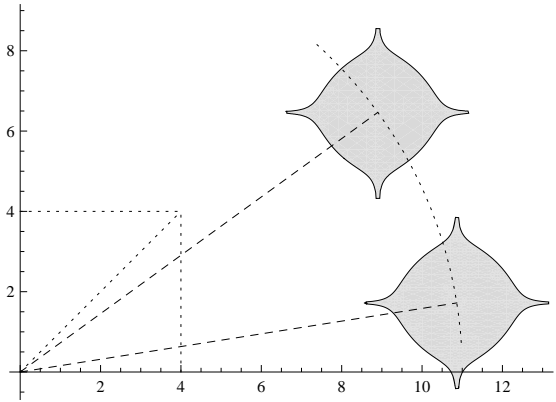}
	\caption{Of the two congruent copies --- $B_{2,-0.4}(1)+11R^{\pi/5}\ee_1$ and $B_{2,-0.4}(1)+11R^{\pi/20}\ee_1$ --- of $B_{2,-0.4}(1)$, the one closer to 
the bisector of the quadrant has a smaller Gaussian measure.}
	\label{fig2}
\end{figure}


Let now $\XX,\XX_1,\XX_2,\dots$ be independent identically distributed random vectors in $\R^k$, with 
a distribution indexed by 
the unknown mean vector $\thb$ of $\XX$;  
let $\E_\thb$ and $\PP_\thb$ denote the corresponding expectation and probability functionals, so that $\E_\thb\XX=\thb$ for all $\thb\in\R^k$. 
Suppose that the covariance matrix 
\begin{equation*}
\Si(\thb):=
\Cov_\thb\XX	
\end{equation*}
is finite, nonsingular, and continuous in $\thb$ in a neighborhood $\mathcal{V}$ of the zero vector $\0$ in $\R^k$. Suppose also that   
\begin{equation*}
\rho_3:=\sup_{\thb\in\mathcal{V}}\E_\thb\|\Si(\thb)^{-1/2}(\XX-\thb)\|^3<\infty; 	
\end{equation*}
here and in what follows, $\|\cdot\|$ denotes the usual Euclidean norm in $\R^k$.

Consider testing the hypothesis $H_0\colon\thb=\0$ versus the alternative $H_1\colon\thb\ne\0$, 
based on the statistic $\Si(\0)^{-1/2}\overline{\XX}_n$, where $\overline{\XX}_n:=\frac1n\sum_{i=1}^n\XX_i$. 
At that, $\Si(\0)$ is supposed to be known. 
Moreover, let us assume that 
\begin{equation*}
	\Si(\0)=I_k,
\end{equation*}
the $k\times k$ identity matrix; this assumption does not diminish generality, since one may replace $\XX_i$ by $\Si(\0)^{-1/2}\XX_i$. 

Consider next tests of the form   
\begin{equation*}
	\de_{n,p}:=\de_{n,p,c}:=\bigii{\sqrt{n}\,\no{\overline{\XX}_n}_p>c}, 
\end{equation*}
where $\ii\cdot$ denotes the indicator function, 
$c\in\R$, $p\in[-\infty,\infty]$, and $\no{\cdot}_p$ is the $p$-mean defined above; let us refer to the tests $\de_{n,p}$ as the $p$-mean tests. 

Take any real numbers $\al$ and $\be$ such that $0<\al<\be<1$. 
Next, define a positive real number $c_p=c_{p,\al}$ as the only root $c$ of the equation $\PP(\no{\ZZ}_p>c)=\al$. 

Finally, take any $\no\cdot_2$-unit vector $\uu\in\R^k$ and let $\s_p=\s_{p;\uu}=\s_{p;\al,\be,\uu}$ be the vector $\s$ in the direction of $\uu$ such that $\PP(\no{\ZZ+\s}_p>c_p)=\be$ -- provided that such a vector $\s$ exists; if it exists, is must be unique. It is easy to see that $\s_p$ exists for all $p\in[0,\infty]$ and all direction vectors $\uu$ -- cf.\ \cite[Proposition~3.8]{d_to_infty}; in particular, $\s_2$ exists for all $\uu$. 

Then an appropriately defined Pitman asymptotic relative efficiency of the $p$-mean tests $(\de_{n,p})$ versus the $2$-mean tests $(\de_{n,2})$ for the alternative mean vectors $\thb\ne\0$ in the direction of a given $\no\cdot_2$-unit vector $\uu\in\R^k$ 
can be expressed as follows (cf.\ 
\cite[Proposition~3.10]{d_to_infty}): 
\begin{equation}\label{eq:a_p2}
\are_{p,2}:=\are_{p,2,\uu}
:=\are_{p,2,\uu}(\al,\be)=\frac{\|\s_2\|^2}{\|\s_p\|^2} 	
\end{equation}
if $\s_p$ exists; otherwise, it is reasonable to let $\are_{p,2}:=0$ -- cf.\ Definitions~2.10 and 3.9 
in \cite{d_to_infty}. 

Based on Corollary~\ref{cor:schur}, we shall prove the following: 

\begin{corollary}\label{cor:are schur2}
$\are_{p,2,\uu}$ is Schur${}^2$-convex in $\uu$ for each $p\in[2,\infty]$ and Schur${}^2$-concave in $\uu$ for each $p\in[-\infty,2]$. 
Therefore, for each $p\in[-\infty,\infty]$, the value of $\are_{p,2,\uu}$ for any $\no\cdot_2$-unit vector $\uu$ lies between the value of $\are_{p,2,\uu}$ for the ``diagonal'' $\no\cdot_2$-unit vector $\one$ and that for ``coordinate'' $\no\cdot_2$-unit vector $\sqrt{k}\,\ee_1$. 
\end{corollary}

Of course, for $k=1$ the $p$-mean is the same for all values of $p$. So, the least nontrivial dimension is $k=2$, in which case the possible values of $\are_{p,2,\uu}$ for $\al=0.05$ and $\be=0.95$ are shown in the left half of Figure~\ref{fig:are}, which illustrates Corollary~\ref{cor:are schur2}. 

\begin{figure}[htbp]
	\centering
	\includegraphics[scale=.65]{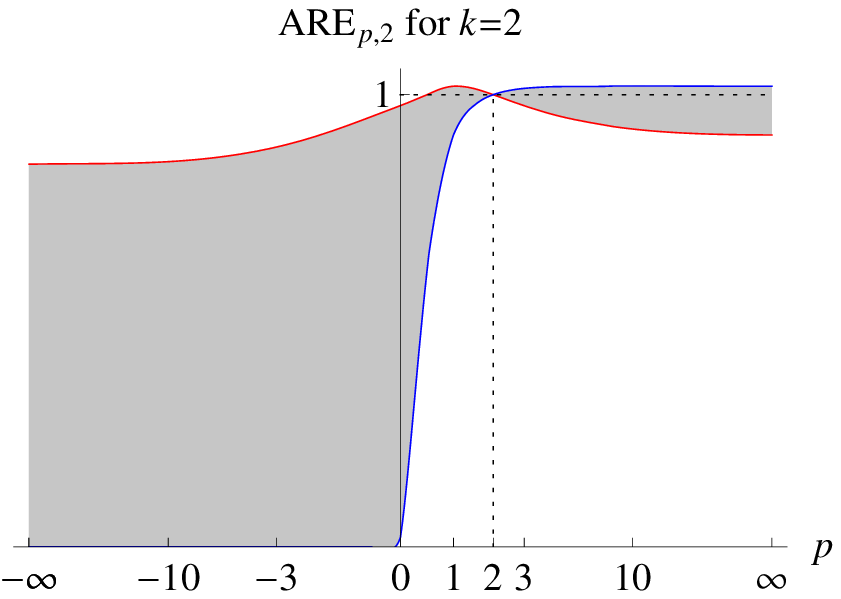}
	\hspace*{.4cm}
	\raisebox{-4pt}{
	\includegraphics[scale=.65]{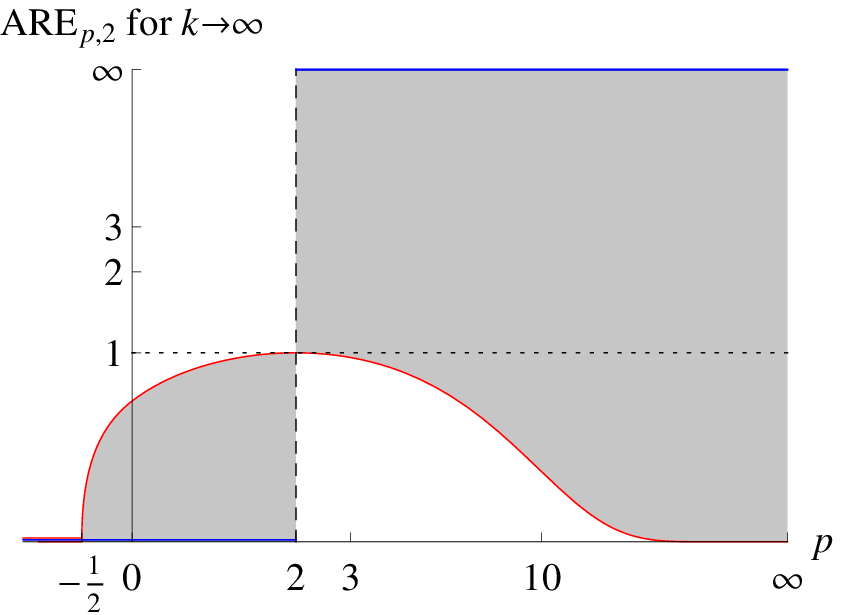}
	}
	\caption{On the left: the values of $\are_{p,2,\uu}$ for $k=2$, $\al=0.05$, $\be=0.95$; 
	on the right: the values of $\are_{p,2,\uu}$ for $k\to\infty$; 
	the horizontal and 
	vertical scales here are nonlinear, with points $(p,\are_{p,2,\uu})$ represented by points  $\big(\psi(p/4),\psi(\are_{p,2,\uu})\big)$ (shaded), where $\psi(x):=2x/(2|x|+3)$, so that $\psi(x)$ increases from $-1$ to $0$ to $1$ as $x$ increases from $-\infty$ to $0$ to $\infty$. 
The values of $\are_{p,2,\uu}$ for the ``diagonal'' directions $\uu$ are represented by the red curve; for the ``coordinate'' $\uu$, by the blue curve for $k=2$ and by the two horizontal blue lines for $k\to\infty$.}
	\label{fig:are}
\end{figure}

Note that for $k=2$ and any given nonzero vector $\uu\in\R^2$ one has $\are_{\infty,2,\uu}=\are_{1,2,R^{\pi/4}\uu}$, where $R^{\pi/4}$ is the operator of rotation through the angle $\pi/4$; \break 
this follows because  $B_1(\vp)=R^{\pi/4}B_\infty(\vp\sqrt2)$. 
In particular, 
$\are_{1,2,(1,1)}=\break
\are_{\infty,2,(\sqrt2,0)}=1.0317\dots$ for $k=2$, $\al=0.05$, and $\be=0.95$. 
It thus appears that, for such $k$, $\al$, and $\be$, the $p$-mean test can at best outperform the LRT by about 3.2\%, which happens for $p=\infty$ and the ``coordinate'' directions, as well as for $p=1$ and the ``diagonal'' directions. 

One may further ask in which directions $\uu$ the $p$-mean test outperforms the LRT test, in the sense that $\are_{p,2,\uu}>1$. It appears (at least for $k=2$) that for each $\uu$ there is some $p$ such that $\are_{p,2,\uu}>1$. 
Indeed, the left half of Figure~\ref{fig:d=2} suggests that (again for $k=2$, $\al=0.05$, and $\be=0.95$) in almost all directions $\uu$ either $\are_{2.1,2,\uu}>1$ or $\are_{1.9,2,\uu}>1$. 
\begin{figure}[htbp]
	\centering
\includegraphics[scale=0.5]
{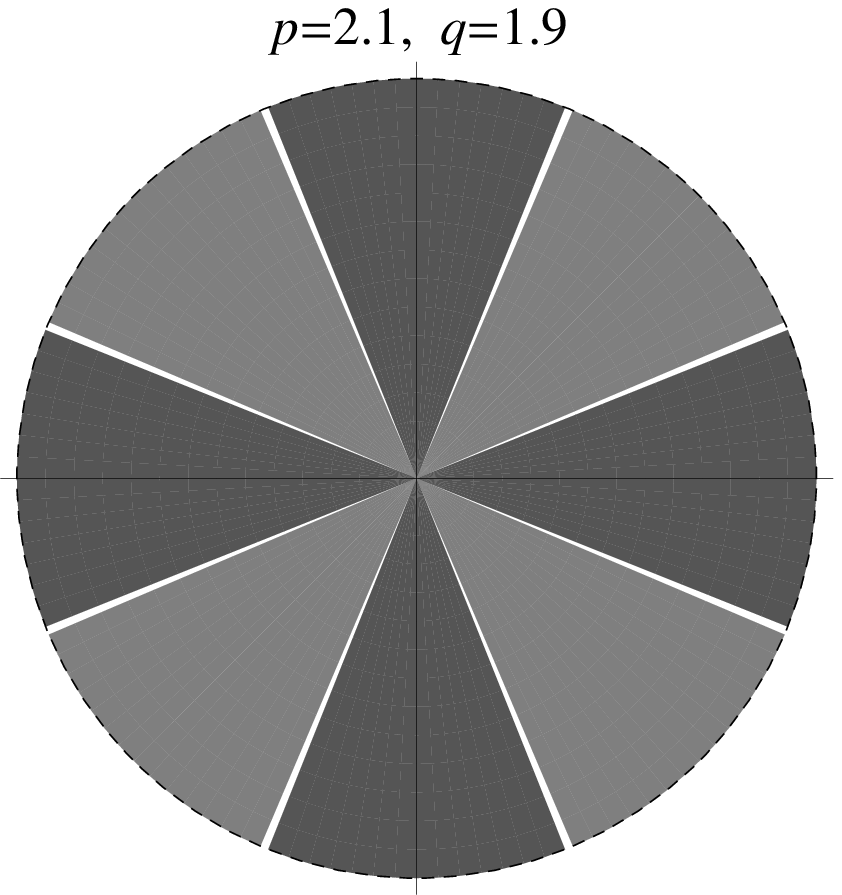}
\hspace*{1cm}
\includegraphics[scale=0.5]
{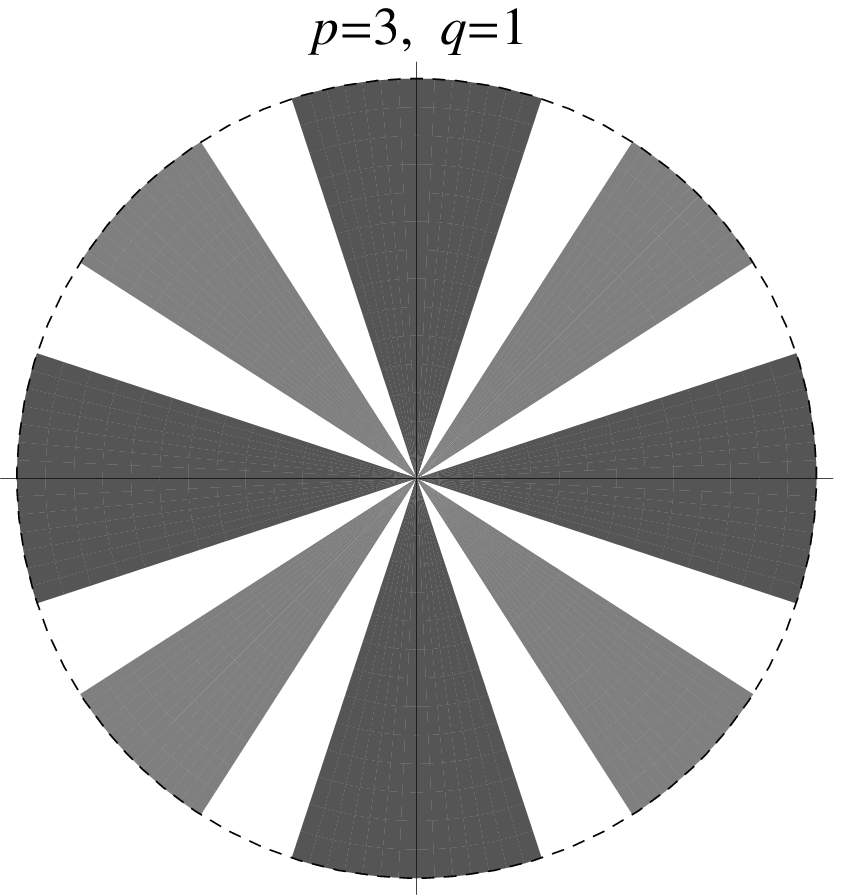}
	\caption{The sectors of directions $\uu$ for $k=2$ where $\are_{p,2,\uu}>1$ (dark\-er gray) 
	and $\are_{q,2,\uu}>1$ (lighter gray); in each of the 8 narrow white sectors, one has $\are_{p,2,\uu}\le1$ and $\are_{q,2,\uu}\le1$. }
	\label{fig:d=2}
\end{figure}
However, at that the improvement in performance is small: $\are_{2.1,2,(\sqrt2,0)}=1.00429\dots$ and  $\are_{1.9,2,(1,1)}=1.00459\dots$. 
Recall also that the maximum improvement \big(for $k=2$, $\al=0.05$, and $\be=0.95$, over all directions $\uu$ and all $p$\big) of the $p$-mean tests over the LRT test appears to be less than 3.2\%. 

The matter is quite different for large $k$, as shown in the right half of Figure~\ref{fig:are}. 
As one can see, for each $p\in[-\infty,2)$ all possible values of $\are_{p,2,\uu}$ \big(for $k\to\infty$, when the dependence of $\are_{p,2,\uu}$ on $\al$ and $\be$ disappears\big) are less than $1$, for all directions $\uu$ of the alternative vector $\th$; that is, for any $p\in[-\infty,2)$ the $2$-mean, LRT test is always asymptotically better than the $p$-mean test.  
However, for each $p\in(2,\infty]$ values of $\are_{p,2,\uu}$ can be arbitrarily large if the vector $\uu$ is far enough from any ``diagonal'' one. For instance, for $p=3$ the possible values of $\are_{p,2,\uu}$ range from $\approx0.96$ to $\infty$.   
This suggests that (say) the $3$-mean test should generally be preferred to the $2$-mean test -- especially, when the direction of the alternative vector $\thb$ is unknown or, even more so, when the direction of $\thb$ is known to be far from a ``diagonal'' one. One can say that the $3$-mean test is much more robust than the LRT with respect to a few relatively large coordinates of an alternative mean vector $\thb$, while the asymptotic efficiency of the $3$-mean test relative to the LRT never falls below $\approx96\%$.  

However, if the dimension $k$ is fixed while $\al\to0$ and $\be\to1$, then the LRT asymptotically outperforms the $p$-mean test for all $p\in[-\infty,\infty]$: 

\begin{corollary}\label{cor:are<1}
For any $p\in[-\infty,\infty]$ and any $\no\cdot_2$-unit vector $\uu\in\R^k$, 
\begin{equation*}
	\limsup\nolimits_{\al\downarrow0,\,\be\uparrow1}\are_{p,2,\uu}(\al,\be)\le1. 
\end{equation*}
\end{corollary}

\section{Proofs}\label{proofs}

\begin{proof}[Proof of Theorem~\ref{th:schur}]\ 

(I) 
By re-scaling, without loss of generality (w.l.o.g.), $\si=1$. Also w.l.o.g, $k\ge2$. Assume now that $k\ge3$. For a moment, take any $\xx$ and $\yy$ in $\R^k$ such that $\xx^2\succ\yy^2$; we have to show that then $\E f(\ZZ+\xx)\le \E f(\ZZ+\yy)$.  

A well-known result by Muirhead \cite{muirRJ} (see, e.g., \cite[Remark~B.1.a of Chapter 2]{marsh-ol}) states that, for any vectors $\mathbf{a}$
and $\mathbf{b}$ in $\mathbb{R}^{n}$ such that $\mathbf{a}\succ\mathbf{b}$,
there exist some $m\in\{1,\dots,n\}$ and vectors $\mathbf{a}^{(0)},%
\mathbf{a}^{(1)},\ldots,\mathbf{a}^{(m)}$ in $\mathbb{R}^{n}$ such that 
$\mathbf{a=a}^{(0)}\succ\mathbf{a}^{(1)}\succ\cdots\succ%
\mathbf{a}^{(m)}=\mathbf{b}$ and for each $j\in\left\{ 1,\ldots,m\right\} $
the vectors $\mathbf{a}^{(j-1)}$ 
and $\mathbf{a}^{(j)}$ differ exactly in two
coordinates. 

On the other hand, by the Fubini theorem, 
\begin{equation*}
	\E f(\ZZ+\xx)=\int_{\R^{k-2}}\!\!\!\E f_{z_3+x_3,\dots,z_k+x_k}(Z_1+x_1,Z_2+x_2)\,
	\prod_{j=3}^k\PP(Z_j\in\dd z_j), 
\end{equation*}
where $f_{u_3,\dots,u_k}(u_1,u_2):=f(u_1,u_2,u_3,\dots,u_k)$, which is Schur${}^2$-concave in \break $(u_1,u_2)\in\R^2$ for each point $(u_3,\dots,u_k)\in\R^{k-2}$. 

Thus (using also the $G_k$-symmetry of the function $f$, as described in Remark~\ref{rem:symm}), w.l.o.g.\ $k=2$. 

So, it suffices to verify that 
$
	\E f(\ZZ+\xx_t)
$ 
is nondecreasing in $t\in[0,\frac\pi4]$, 
where 
\begin{equation*}
	\xx_t:=R^t\xx_0=(r\cos t,r\sin t), 
\end{equation*} 
$\xx_0:=(r,0)$, $r\in(0,\infty)$, and $R^t$ is the operator of the rotation through angle $t$, so that 
$R^t\xx=(x_1\cos t-x_2\sin t,\,x_1\sin t+x_2\cos t)$ for any $\xx=(x_1,x_2)\in\R^2$.  
Introducing now the set
\begin{gather*}
	A_1:=\{(u,v)\in\R^2\colon 0<v<u\}, 
\end{gather*}
and again using the $G_k$-symmetry of the function $f$ (for $k=2$), one can write 
\begin{align*}
	\E f(\ZZ+\xx_t)&=\int_{\R^2} f(\yy)\,\frac1{2\pi}\,e^{-\|\yy-R^t\xx_0\|^2/2}\,\dd\yy \\
	&=\frac1{2\pi}\,\int_{\R^2} f(\yy)\,e^{-(\|\yy\|^2+r^2)/2}\,e^{\yy\cdot R^t\xx_0}\,\dd\yy \\
	&=\frac1{2\pi}\,\sum_{g\in G_2}\int_{A_1} f(\yy)\,e^{-(\|\yy\|^2+r^2)/2}\,e^{(g\yy)\cdot R^t\xx_0}\,\dd\yy.  
\end{align*}

Consider now any point $\yy\in A_1$, so that $\yy=(\rho\cos\th,\rho\sin\th)$ for some $\rho\in(0,\infty)$ and $\th\in(0,\frac\pi4)$; then 
\begin{gather*}
	\big\{(g\yy)\cdot R^t\xx_0\colon g\in G_2\big\}
	=\big\{\rho r\cos(\th_i-t)\colon i=1,\dots,8\big\} \\
	\begin{aligned}
	=\big\{&\rho r\,\cos(\th-t),\;\rho r\,\sin(\th+t),\;-\rho r\,\sin(\th-t),\;-\rho r\,\cos(\th+t),\; \\
	&\rho r\,\cos(\th+t),\;\rho r\,\sin(\th-t),\;-\rho r\,\sin(\th+t),\;-\rho r\,\cos(\th-t)\big\}, 
	\end{aligned}
\end{gather*}
where 
\begin{align*}
	\{\th_1,\dots,\th_8\}=&\{ \th,\;\tfrac\pi2-\th,\;\tfrac\pi2+\th,\;\pi-\th,\\
														& -\th,\;\th-\tfrac\pi2,\;-\tfrac\pi2-\th,\;\th-\pi\}. 
\end{align*}
So, 
\begin{align}
	\E f(\ZZ+\xx_t)&=\frac1\pi\,\int_{A_1} e^{-(\rho^2+r^2)/2}\,f_\rho(\th)\,[w_\rho(\th+t)+w_\rho(\th-t)]\,\rho\,\dd\rho\,\dd\th, \notag\\
	&=\frac1\pi\,\int_0^\infty\,\rho\,e^{-(\rho^2+r^2)/2}\,I_\rho(t)\,\dd\rho, \label{eq:qq}	
\end{align}
where 
\begin{align}
f_\rho(\th)&:=f(\rho\cos\th,\rho\sin\th), \notag\\
	w_\rho(u)&:=\cosh(\rho r\cos u)+\cosh(\rho r\sin u), \label{eq:w} \\
	I_\rho(t)&:=\int_0^{\frac\pi4}f_\rho(\th)\,[w_\rho(\th+t)+w_\rho(\th-t)]\,\dd\th \notag\\
	&=\int_{(0,\pi/4]}J_{\rho,\eta}(t)\,\dd f_\rho(\eta) + J_{\rho,\eta}(0)\,f_\rho(0), \label{eq:I}\\
	J_{\rho,\eta}(t)&:=\int_\eta^{\frac\pi4}[w_\rho(\th+t)+w_\rho(\th-t)]\,\dd\th  \notag \\
	&=\int_{\eta+t}^{\frac\pi4+t}w_\rho(u)\,\dd u+\int_{\eta-t}^{\frac\pi4-t}w_\rho(u)\,\dd u;  \notag
\end{align}
the equality \eqref{eq:I} holds by the Fubini theorem --- because, by virtue of the Schur${}^2$-concavity of $f$, the function $f_\rho$ is nondecreasing on $[0,\frac\pi4]$ (for each $\rho\in(0,\infty)$), and so, $f_\rho(\th)=f_\rho(0)+\int_{(0,\th]}\dd f_\rho(\eta)$ for almost all $\th\in[0,\frac\pi4]$ \big(in fact, for all $\th\in[0,\frac\pi4]$ except possibly $\th$ in the at most countable set of the discontinuities of $f_\rho$ on $[0,\frac\pi4]$\big). 
Observe next that 
\begin{align*}
J_{\rho,\eta}'(t):=\tfrac{\partial}{\partial t}J_{\rho,\eta}(t)
&=w_\rho(\tfrac\pi4+t)-w_\rho(\eta+t)-w_\rho(\tfrac\pi4-t)+w_\rho(\eta-t) \\
&=w_\rho(\eta-t)-w_\rho(\eta+t) 
\end{align*}
by \eqref{eq:w}, 
since $\cos(\tfrac\pi4\pm t)=\sin(\tfrac\pi4\mp t)$ for all $t\in\R$. 
So, 
\begin{align*}
J_{\rho,\eta}'(t)=\sum_{j=2}^\infty\frac{(\rho r)^{2j}}{(2j)!}\,[s_j(|\eta-t|)-s_j(\eta+t)],
\end{align*}
where
\begin{equation*}
	s_j(\al):=\cos^{2j}\al+\sin^{2j}\al, 
\end{equation*}
which is easily seen to be increasing in $|\al-\frac\pi4|\in[0,\frac\pi4]$ for each $j=2,3,\dots$. 
On the other hand, 
\begin{equation*}
	0\le|\eta+t-\tfrac\pi4|<\big||\eta-t|-\tfrac\pi4\big|\le\tfrac\pi4
\end{equation*}
for all $\eta$ and $t$ in $(0,\frac\pi4)$. 

We conclude that $J_{\rho,\eta}'(t)>0$ for all $t\in(0,\frac\pi4)$, so that $J_{\rho,\eta}(t)$ is increasing in  $t\in[0,\frac\pi4]$. 
So, by \eqref{eq:I}, $I_\rho(t)$ is nondecreasing in  $t\in[0,\frac\pi4]$ \big(in fact, $I_\rho(t)$ is strictly increasing in $t\in[0,\frac\pi4]$ unless $f_\rho(\th)=f(\rho\cos\th,\rho\sin\th)$ is constant in  $\th\in(0,\frac\pi4]$\big). 

Thus, in view of \eqref{eq:qq} and \eqref{eq:I}, $\E f(\ZZ+\xx_t)$ is indeed nondecreasing in $t\in[0,\frac\pi4]$. 
This proves part (I) of the theorem. 

(II) Suppose that the function $f^{(\si)}$ is not strictly Schur${}^2$-concave. Then, by the just proved part (I) of the theorem and the definition of the strict Schur${}^2$-concavity, there exist $\aa$ and $\bb$ in $\R^k$ such that $\aa^2\succ\bb^2$ and $f^{(\si)}(\aa)=f^{(\si)}(\bb)<\infty$; fix any such $\aa$ and $\bb$. 
For a moment, fix also any $\vp\in(0,\si)$. 
W.l.o.g., $k\ge2$. 

Let $\vpi_\si$ denote the probability density function of $\si\ZZ$, so that 
\begin{equation*}
	\vpi_\si(\uu)=\Big(\frac1{\si\sqrt{2\pi}}\Big)^k\,\exp\Big(-\frac{\|\uu\|^2}{2\si^2}\Big) 
\end{equation*}
for all $\uu\in\R^k$. 
Then
\begin{equation*}
	f^{(\vp)}(\xx)=\int_{\R^k} f(\uu+\aa)\,\vpi_\si(\uu)\,\frac{\vpi_\vp(\aa-\xx+\uu)}{\vpi_\si(\uu)}\,\dd\uu
\end{equation*}
for all $\xx\in\R^k$, and the ratio $\frac{\vpi_\vp(\aa-\xx+\uu)}{\vpi_\si(\uu)}$ is continuous in $\xx$ and bounded over all $\xx$ in any bounded subset of $\R^k$ and all $\uu\in\R^k$. 
Also, 
\begin{equation}\label{eq:finite}
	\int_{\R^k} f(\uu+\aa)\,\vpi_\si(\uu)\dd\uu = f^{(\si)}(\aa)<\infty.
\end{equation} 
So, the function $f^{(\vp)}$ is finite and, by dominated convergence, continuous on $\R^k$. 
It also follows from \eqref{eq:finite} that  
$f$ is locally integrable and hence finite almost everywhere on $\R^k$. 

Next, note the semigroup property of the operator family $T^{\si^2}\colon f\mapsto f^{(\si)}$; namely, $T^{\si^2}=T^{\si^2-\vp^2}T^{\vp^2}$, so that  
\begin{equation*}
	f^{(\si)}(\aa)=\E f^{(\vp)}\big(\sqrt{\si^2-\vp^2}\,\ZZ+\aa\big). 
\end{equation*}
At that, by part (I) of the theorem, the function $f^{(\vp)}$ is Schur${}^2$-concave and, in particular, has the group symmetry property. 

Re-tracing now the lines of the proof of part (I) --- and, especially, the penultimate paragraph of that proof ---  with $f^{(\vp)}$ and $\sqrt{\si^2-\vp^2}$ in place of $f$ and $\si$, respectively, one finds that, for both of the conditions $\aa^2\succ\bb^2$ and $f^{(\si)}(\aa)=f^{(\si)}(\bb)<\infty$ to hold, it is necessary that 
$f^{(\vp)}(\rho\cos\th,\rho\sin\th,u_3,\dots,u_k)$ be constant in $\th\in(0,\frac\pi4]$ for almost any given  $(\rho,u_3,\dots,u_k)\in(0,\infty)\times\R^{k-2}$; hence, by the noted continuity and group symmetry of $f^{(\vp)}$, the expression $f^{(\vp)}(\rho\cos\th,\rho\sin\th,u_3,\dots,u_k)$ must be constant in all $\th\in[0,2\pi]$ for each \break 
$(\rho,u_3,\dots,u_k)\in(0,\infty)\times\R^{k-2}$. 
In other words, for 
\begin{equation*}
	h(v_1,\dots,v_k):=h^{(\vp)}(v_1,\dots,v_k):=f^{(\vp)}(\sqrt{v_1},\dots,\sqrt{v_k})
\end{equation*}
the assertion $\A_2$ holds, if 
for each $j\in\{2,\dots,k\}$ one defines $\A_j$ as the following assertion: 
\begin{center}
\begin{parbox}{4.2in}
{for any $(u_1,\dots,u_k)$ and $(v_1,\dots,v_k)$ in $[0,\infty)^k$ such that $u_1+\dots+u_j\break
=v_1+\dots+v_j$ and $u_i=v_i$ for all natural $i\in[j+1,k]$, one has $h(u_1,\dots,u_k)=h(v_1,\dots,v_k)$. 
}
\end{parbox}
\end{center}
%
%
Take now any natural $j\in[2,k-1]$ (if such a number $j$ exists) and assume that $\A_j$ holds. Take then any $(u_1,\dots,u_k)$ and $(v_1,\dots,v_k)$ in $[0,\infty)^k$ such that $u_1+\dots+u_{j+1}=v_1+\dots+v_{j+1}$ and $u_i=v_i$ for all natural $i\in[j+2,k]$. We want to show that then $h(u_1,\dots,u_k)=h(v_1,\dots,v_k)$. W.l.o.g.\ $u_{j+1}\le v_{j+1}$. Replace now $(v_1,\dots,v_k)$ by 
\begin{align*}
(\tilde v_1,\dots,\tilde v_k) &:=(v_1+v_{j+1}-u_{j+1},v_2,\dots,v_j,u_{j+1},v_{j+2},\dots,v_k) \\
&=(v_1+v_{j+1}-u_{j+1},v_2,\dots,v_j,u_{j+1},u_{j+2},\dots,u_k). 
\end{align*} 
Then, by $\A_2$ and the permutation symmetry of $h$, one has $h(\tilde v_1,\dots,\tilde v_k)=h(v_1,\dots,v_k)$, since $\tilde v_1+\tilde v_{j+1}=v_1+v_{j+1}$. On the other hand, $\A_j$ implies $h(\tilde v_1,\dots,\tilde v_k)=h(u_1,\dots,u_k)$. So, it is proved by induction that $\A_k$ holds. 

This means that the function $f^{(\vp)}$ is spherically symmetric, for each $\vp\in(0,\si)$. Now we are going to use the following approximation property of the convolution:
\begin{equation}\label{eq:appr}
	f^{(\vp)}\to f\quad\text{as $\vp\downarrow0$ almost everywhere on $\R^k$. }
\end{equation}
Variants of this property are quite well known, under different conditions on $f$; cf.\ e.g.\ \cite[Theorem~8.15]{folland}. However, as it is often the case in such situations, it appears easier to prove \eqref{eq:appr} than to find a quite ready to use result. We shall provide a proof of \eqref{eq:appr} a little later. 

Since $f^{(\vp)}$ is spherically symmetric for each $\vp\in(0,\si)$, \eqref{eq:appr} implies that there exist a subset (say $\Theta$) of $[0,\infty)$ of full Lebesgue measure and a  function $F\colon[0,\infty)\to[0,\infty)$ such that 
$f^{(\vp)}(\xx)\to F(\|\xx\|)$ as $\vp\downarrow0$ for each $\xx\in\R^k$ such that $\|\xx\|\in\Theta$. 
For all $\xx\in\R^k$, let now $\tilde f(\xx):=F(\|\xx\|)$ if $\|\xx\|\in\Theta$ and $\tilde f(\xx):=0$ otherwise. Then the function $\tilde f$ is spherically symmetric and, by \eqref{eq:appr}, $f$ is equivalent to $\tilde f$.  

It remains to prove \eqref{eq:appr}. Take any Lebesgue point $\xx$ of the function $f$ and also take any $\g\in(0,\infty)$, so that there exists some $\de_0=\de_0(\xx,\g)\in(0,\infty)$ such that 
\begin{equation}\label{eq:Leb}
	\Big|\int_{B_\de}\big(f(\xx+\uu)-f(\xx)\big)\dd\uu\Big|\le\g\int_{B_\de}\dd\uu
	\quad\text{for all $\de\in(0,\de_0]$,}
\end{equation}
where $B_\de:=\{\uu\in\R^k\colon\|\uu\|<\de\}$, the $\de$-ball in $\R^k$. 
Next, write 
\begin{align*}
	|f^{(\vp)}(\xx)-f(\xx)|&\le I_{1,\vp}+I_{2,\vp},\quad\text{where}\\
	I_{1,\vp}&:=\Big|\int_{B_{\de_0}}\big(f(\xx+\uu)-f(\xx)\big)\,\vpi_\vp(\uu)\dd\uu\Big|, \\
	I_{2,\vp}&:=\int_{\R^k\setminus B_{\de_0}}\big(f(\xx+\uu)+f(\xx)\big)\,\vpi_\vp(\uu)\dd\uu
\end{align*}
(recall that the values of $f$ are in $[0,\infty]$). 
Observe that, for each $\uu\in\R^k$ with $\|\uu\|>\de_0$, the expression $\vpi_\vp(\uu)$ monotonically decreases to $0$ as $\vp$ decreases from $\de_0/\sqrt k$ to $0$. 
So, by dominated convergence, 
$
	I_{2,\vp}\to0\quad\text{as $\vp\downarrow0$.}
$ 

To bound $I_{1,\vp}$, observe that, for each $c\in\R$, the set $\{\uu\in\R^k\colon\vpi_\vp(\uu)>c\}$ is the (possibly empty) ball $B_{\la(\vp,c)}$ of a certain (possibly zero) radius $\la(\vp,c)$. Hence, in view of \eqref{eq:Leb}, 
\begin{align*}
I_{1,\vp}&=\Big|\int_{B_{\de_0}}\big(f(\xx+\uu)-f(\xx)\big)\dd\uu\,\int_0^\infty\ii{\vpi_\vp(\uu)>c}\dd c\Big| \\
&=\Big|\int_0^\infty\dd c \int_{B_{\de_0\wedge\la(\vp,c)}}\big(f(\xx+\uu)-f(\xx)\big)\dd\uu\Big| \\
&\le\g\,\int_0^\infty\dd c \int_{B_{\de_0\wedge\la(\vp,c)}}\dd\uu \\
&=\g\,\int_{B_{\de_0}}\dd\uu\,\int_0^\infty\ii{\vpi_\vp(\uu)>c}\dd c \\
&=\g\,\int_{B_{\de_0}}\vpi_\vp(\uu)\dd\uu
\,\le\,\g. 
\end{align*}
Thus, $\limsup_{\vp\downarrow0}|f^{(\vp)}(\xx)-f(\xx)|\le\g$. Since $\g\in(0,\infty)$ was chosen arbitrarily, one has $f^{(\vp)}(\xx)\to f(\xx)$ as $\vp\downarrow0$, for each Lebesgue point $\xx$ of the locally integrable function $f$. 
It remains to recall that the Lebesgue set of any locally integrable function is of full Lebesgue measure (see e.g.\ \cite[Theorem~3.20]{folland}). 
\end{proof}


\begin{proof}[Proof of Proposition~\ref{prop:counterex}]
Take any integer $k\ge2$. Take then any $\vp\in(0,\sqrt k-1)$ and let $R:=\frac{\vp^2+k-1}{2\vp}$ and $r:=R-1-\vp=\frac{k-(1+\vp)^2}{2\vp}>0$. Introduce also the vectors $\xx_0:=r\ee_1$ and $\xx_1:=\frac r{\sqrt k}\,\one$. 
Let $\XX$ be any random vector uniformly distributed in the Euclidean ball $B(R):=\{\xx\in\R^k\colon\|\xx\|\le R\}$. 
Finally, let $A:=B_\infty(1)=\{\xx\in\R^k\colon\max_j|x_j|\le1\}$. 
Then $\XX$ has a spherically symmetric unimodal distribution and $A$ is Schur${}^2$-convex, by Proposition~\ref{prop:means}. 

Next, for any $\yy\in A$  
\begin{equation*}
	\|\yy+\xx_0\|^2=(y_1+r)^2+y_2^2+\dots+y_k^2\le(1+r)^2+k-1=(R-\vp)^2+k-1=R^2. 
\end{equation*}
It follows that $A+\xx_0\subseteq B(R)$, whence 
\begin{equation*}
	\PP(\XX\in A+\xx_0)=\frac{\mes(A+\xx_0)}{\mes B(R)}=\frac{\mes A}{\mes B(R)}, 
\end{equation*}
where $\mes$ stands for the Lebesgue measure in $\R^k$. 
One the other hand, 
\begin{equation*}
	\PP(\XX\in A+\xx_1)<\frac{\mes A}{\mes B(R)}, 
\end{equation*}
because the point $(1+\frac r{\sqrt k})\one$ is a vertex of the shifted cube $A+\xx_1$, and the distance from this point to the ball $B(R)$ is $\sqrt k+r-R=\sqrt k-1-\vp>0$. 
So, $\PP(\XX\in A+\xx_1)<\PP(\XX\in A+\xx_0)$ while $\xx_1^2\preceq\xx_0^2$. 
Thus, the function 
$
\R^k\ni\xx\mapsto\PP(\XX\in A+\xx) 	
$
is not Schur${}^2$-concave. 
This proof is illustrated by Figure~\ref{fig:counterex}. 
\begin{figure}%
\includegraphics[scale=.5]{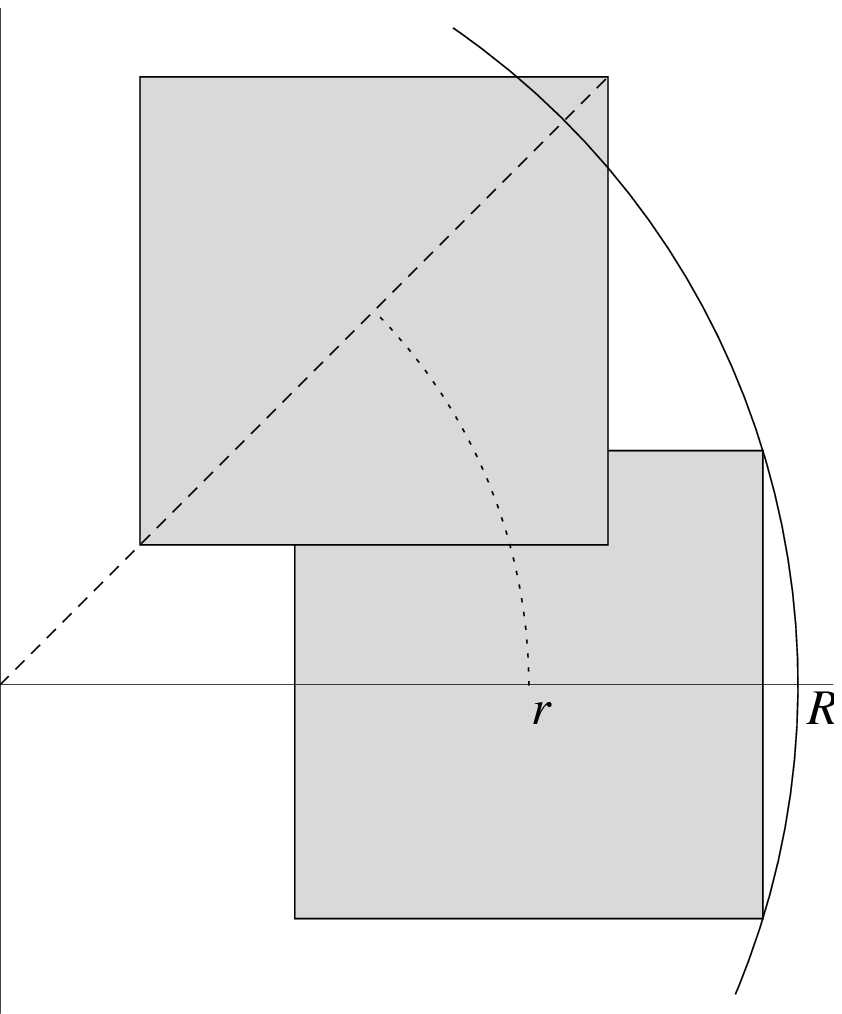}%
\caption{Illustration of the proof of Proposition~\ref{prop:counterex} (for $k=2$ and $\vp=0.15$, whence $R\approx3.41$ and $r\approx2.26$. Part of the square $A+\xx_1$ is seen to be outside the ball $B(R)$, while the square $A+\xx_0$ is entirely contained in $B(R)$. 
Since $\XX$ is uniformly distributed in $B(R)$, it follows that $\PP(\XX\in A+\xx_1)<\PP(\XX\in A+\xx_0)$.
}%
\label{fig:counterex}%
\end{figure}
\end{proof}

\begin{proof}[Proof of Proposition~\ref{prop:pq-means}]
Observe that, for any $\uu\in(0,\infty)^k$, any $i\in\{1,\dots,k\}$, and any real $p$ and $q$ such that $p>q$,  the partial derivative of $\no{(\sqrt{u_1},\dots,\sqrt{u_k})}_{p,q}$ in $u_i$ equals  $pu_i^{-1+p/2}\sum_ju^{q/2}-qu_i^{-1+q/2}\sum_ju^{p/2}$ up to a factor which does not depend on $i$. Now the proof is easy to complete, on recalling the Schur theorem that, for $f(\uu)=f(u_1,\dots,u_k)$ to be Schur-convex in $(u_1,\dots,u_k)$, it is enough that $f(\uu)$ be symmetric in $(u_1,\dots,u_k)$ and 
$(\frac{\partial}{\partial u_i}-\frac{\partial}{\partial u_j})f(\uu)$ be equal in sign to $u_i-u_j$ -- see e.g.\ \cite[Theorem A.4 of Chapter 3]{marsh-ol}. 
\end{proof}

\begin{proof}[Proof of Proposition~\ref{prop:hat,check B}] 
Observe that a point $\xx\in[0,\infty)^k$ is in the set $\hat B_p(a,\vp)$ iff $\sum_i|\sqrt u_i-a|^p\le k\vp^p$, where $u_i:=x_i^2$. 
Next, 
the function $u\mapsto|\sqrt u-a|^p$ is convex on $[0,\infty)$ given any $p\ge2$ and any $a\ge0$. So, by virtue of the $G_k$-symmetry and the Hardy-Littlewood-Polya theorem \cite[Proposition~B.1 of Chapter~4]{marsh-ol}, the set $\hat B_p(a,\vp)$ is Schur${}^2$-convex for $p\ge2$. 

It remains to show that the sets $\check B_p(a,\vp)$ are Schur${}^2$-concave for $p\in[1,2]$. 
By the 
$G_k$-symmetry and the mentioned 
result by Muirhead (\cite[Remark~B.1.a of Chapter 2]{marsh-ol}), it suffices to show that the function $[\frac r{\sqrt2},r]\ni x\mapsto|x-a|^p+(r^2-x^2)^{p/2}$ is nonincreasing, given any $p\in[1,2]$, $r>0$, and $a\ge0$. But the derivative of this function equals in sign $|x-a|^{p-1}\sign(x-a)-xy^{p-2}\le |x-a|^{p-1}\sign(x-a)-x^{p-1}\le0$, where $y:=(r^2-x^2)^{1/2}\le x$. 
\end{proof}

In the proof of Corollary~\ref{cor:are schur2}, we shall use 

\begin{lemma}\label{lem:mono}
For any given $p\in[-\infty,\infty]$, $\vv\in\R^k\setminus\{\0\}$, and $c\in(0,\infty)$, the function 
$
[0,\infty)\ni t\longmapsto\PP(\no{\ZZ+t\vv}_p>c) 	
$ 
is continuously and strictly increasing. 
\end{lemma}

\begin{proof}[Proof of Lemma~\ref{lem:mono}]
By \cite[Theorem\ 2.2]{jog77}, the function 
$
[0,\infty)\ni t\longmapsto\PP(\no{\ZZ+t\vv}_p>c) 	
$ 
is nondecreasing; working slightly harder, one can see that this function is strictly increasing. 
The continuity follows by dominated convergence, because the intersection of the set $\{\zz\in\R^d\colon\no{\zz+\s}_p\le c\}$ with each straight line in $\R^d$ parallel to a coordinate axis is a (possibly empty) interval with endpoints continuously depending on $\s$.
\end{proof}

\begin{proof}[Proof of Corollary~\ref{cor:are schur2}]\ 
Consider the case $p\in[-\infty,2]$. Assume, to the contrary, that for some $\no\cdot_2$-unit vectors $\uu$ and $\vv$ in $\R^k$ one has $\are_{p,2,\uu}<\are_{p,2,\vv}$ while $\uu^2\preceq\vv^2$. In particular, this implies that $\are_{p,2,\vv}>0$. So, by the last sentence preceding the statement of Corollary~\ref{cor:are schur2}, the vector $\s_{p;\vv}$ exists. 

Consider first the subcase when the vector $\s_{p;\uu}$ exists as well. Then for each $r\in\{2,p\}$ and each $\ww\in\{\uu,\vv\}$ there exists a unique $t_{r,\ww}\in(0,\infty)$ such that $\s_{r,\ww}=t_{r,\ww}\ww$, 
so that $\PP(\no{\ZZ+t_{r,\ww}\ww}_r>c_r)=\be$.  
By the spherical symmetry of the distribution of $\ZZ$, one has $t_{2,\uu}=t_{2,\vv}$ or, equivalently, $\|\s_{2,\uu}\|=\|\s_{2,\vv}\|$. 
So, the assumption $\are_{p,2,\uu}<\are_{p,2,\vv}$ implies that  
$\|\s_{p,\vv}\|<\|\s_{p,\uu}\|$ or, equivalently, $t_{p,\vv}<t_{p,\uu}$. 
Therefore, by Corollary~\ref{cor:means2} and
Lemma~\ref{lem:mono},  
\begin{equation*}
\begin{aligned}
\be=\PP(\no{\ZZ+t_{p,\uu}\uu}_p>c_p)
\ge\PP(\no{\ZZ+t_{p,\uu}\vv}_p>c_p)
>\PP(\no{\ZZ+t_{p,\vv}\vv}_p>c_p)=\be, 	
\end{aligned}	
\end{equation*}
a contradiction. 

Consider now the other subcase of the case $p\in[-\infty,2]$, when the vector $\s_{p;\uu}$ does not exist. Then $\PP(\no{\ZZ+t\uu}_p>c_p)<\be$ for all $t\in(0,\infty)$ --- because, by Lemma~\ref{lem:mono}, the probability $\PP(\no{\ZZ+t\uu}_p>c_p)$ is continuous in $t$ and equals $\al$ at $t=0$. 
In particular, $\PP(\no{\ZZ+t_{p,\vv}\uu}_p>c_p)<\be$, which is a contradiction, since, again by Corollary~\ref{cor:means2}, $\PP(\no{\ZZ+t_{p,\vv}\uu}_p>c_p)
\ge\PP(\no{\ZZ+t_{p,\vv}\vv}_p>c_p)=\be$. 

The case $p\in[2,\infty]$ is similar to the case $p\in[-\infty,2]$, and even a bit simpler, because in this case the vector $\s_{p;\uu}$ always exists. 
\end{proof}

\begin{proof}[Proof of Corollary~\ref{cor:are<1}]\ 
Consider first the case $p\in[2,\infty)$. As on page~\pageref{ell_p} of the introduction, let $\|\cdot\|_p$ denote the $\ell_p$ norm in $\R^k$. 
Let $\al\downarrow0$ and $c:=k^{1/p}c_{p,\al}$, so that $\al=\PP(\|\ZZ\|_p>c)\ge\PP(|Z_1|>c)$ and hence $c\to\infty$ and $\PP(|Z_1|>c)=e^{-c^2/(2+o(1))}$, which implies $\al\ge e^{-c^2/(2+o(1))}$ and hence $c\gsim\sqrt{2\ln\frac1\al}$; we write $a\lsim b$ or, equivalently, $b\gsim a$ if $\limsup\frac ab\le1$. 
By Corollary~\ref{cor:are schur2} and \eqref{eq:one<u<e}, without loss of generality 
$\uu=\sqrt k\,\ee_1$; let then $\s=(s,0,\dots,0)\in\R^k$ stand for the corresponding vector $\s_{p;\uu}=\s_{p;\al,\be,\uu}$, which exists, since in this case $p\in[2,\infty)$ and hence $p\in[0,\infty]$. Then $s>0$ and, as $\al\downarrow0$ and $\be\uparrow1$,  
\begin{align}
	1-\be&=\PP(\|\ZZ-\s\|_p\le c) \notag\\
	&\ge\PP(|Z_1-s|\le c-1)\PP\big(|Z_2|^p+\dots+|Z_k|^p\le c^p-(c-1)^p\big) \label{eq:ge}\\
	&\gsim\PP(|Z_1-s|\le c-1)
	\sim\PP\big(Z_1\ge s-(c-1)\big)
	=e^{-(s-c)^2/(2+o(1))}; \notag
\end{align}
here, the relation $\gsim$ is due to the fact that $c\to\infty$ and hence $c^p-(c-1)^p\ge p(c-1)^p\to\infty$, whereas the last two displayed relations, $\sim$ and $=$, follow because $0\leftarrow1-\be\gsim\PP(|Z_1-s|\le c-1)$ and hence $s-(c-1)\to\infty$. 
Thus, one has $s-c\gsim\sqrt{2\ln\frac1{1-\be}}$, and so, 
\begin{equation*}
	\|\s_p\|=s\gsim\sqrt{2\ln\tfrac1\al}+\sqrt{2\ln\tfrac1{1-\be}}. 
\end{equation*}

Moreover, in the special case when $p=2$, for some constant $C(k)$ depending only on $k$ one has $\al=\PP(\|\ZZ\|_p>c)=C(k)\int_{c^2}^\infty u^{k/2-1}e^{-u/2}\dd u=e^{-c^2/(2+o(1))}$ and hence $c\sim\sqrt{2\ln\frac1\al}$; moreover, $1-\be=\PP(\|\ZZ-\s\|_p\le c) 
\le\PP(|Z_1-s|\le c)
	=e^{-(s-c)^2/(2+o(1))}$, whence 
\begin{equation*}
	\|\s_2\|\sim\sqrt{2\ln\tfrac1\al}+\sqrt{2\ln\tfrac1{1-\be}}. 
\end{equation*}
In view of \eqref{eq:a_p2}, this completes the proof for $p\in[2,\infty)$. 

The case $p=\infty$ is similar, with \eqref{eq:ge} replaced by $\PP(\|\ZZ-\s\|_\infty\le c)=\PP(|Z_1-s|\le c)\PP(\max_2^k|Z_j|\le c)$. 

It remains to consider the case $p\in[-\infty,2]$. Let $\al\downarrow0$ and $c:=c_{p,\al}$, so that $\al=\PP(\no\ZZ_p>c)\ge\PP(\min_j|Z_j|>c)=\PP(|Z_1|>c)^k=e^{-kc^2/(2+o(1))}$ and hence $c\gsim \sqrt{\frac2k\,\ln\frac1\al}$. 
By Corollary~\ref{cor:are schur2} and \eqref{eq:one<u<e}, without loss of generality 
$\uu=\one$; let then $\s=(s,\dots,s)\in\R^k$ stand for the corresponding vector $\s_{p;\uu}$, which exists by dominated convergence, 
since $\no{\zz-t\one}_p\to\infty$ as $t\to\infty$, for each $\zz\in\R^k$ and each $p\in[-\infty,\infty]$. 
Then $s>0$ and, as $\al\downarrow0$ and $\be\uparrow1$,  
\begin{multline*}
	1-\be=\PP(\no{\ZZ-\s}_p\le c) 
	\ge\PP(\no{\ZZ-\s}_\infty\le c) \\
	=\PP(|Z_1-s|\le c)^k
	\sim\PP(Z_1>s-c)^k
	=e^{-k(s-c)^2/(2+o(1))}. 
\end{multline*}
Thus, $s-c\gsim\sqrt{\frac2k\,\ln\frac1{1-\be}}$, and so, 
\begin{equation*}
	\|\s_p\|=s\sqrt k\gsim\sqrt{2\ln\tfrac1\al}+\sqrt{2\ln\tfrac1{1-\be}}. 
\end{equation*}
This completes the proof for $p\in[-\infty,2]$ as well. 
\end{proof}


\bibliographystyle{abbrv}


\bibliography{C:/Users/Iosif/Documents/mtu_home01-30-10/bib_files/citations}

\def\polhk#1{\setbox0=\hbox{#1}{\ooalign{\hidewidth
  \lower1.5ex\hbox{`}\hidewidth\crcr\unhbox0}}} \def\cprime{$'$}
\begin{thebibliography}{10}

\bibitem{anderson55}
T.~W. Anderson.
\newblock The integral of a symmetric unimodal function over a symmetric convex
  set and some probability inequalities.
\newblock {\em Proc. Amer. Math. Soc.}, 6:170--176, 1955.

\bibitem{coxeter}
H.~S.~M. Coxeter.
\newblock {\em Regular polytopes}.
\newblock Second edition. The Macmillan Co., New York, 1963.

\bibitem{dharm-joag88}
S.~Dharmadhikari and K.~Joag-Dev.
\newblock {\em Unimodality, convexity, and applications}.
\newblock Probability and Mathematical Statistics. Academic Press Inc., Boston,
  MA, 1988.

\bibitem{folland}
G.~B. Folland.
\newblock {\em Real analysis}.
\newblock Pure and Applied Mathematics (New York). John Wiley \& Sons Inc., New
  York, 1984.
\newblock Modern techniques and their applications, A Wiley-Interscience
  Publication.

\bibitem{gordon02}
Y.~Gordon, A.~Litvak, C.~Sch{\"u}tt, and E.~Werner.
\newblock Geometry of spaces between polytopes and related zonotopes.
\newblock {\em Bull. Sci. Math.}, 126(9):733--762, 2002.

\bibitem{HKP}
R.~L. Hall, M.~Kanter, and M.~D. Perlman.
\newblock Inequalities for the probability content of a rotated square and
  related convolutions.
\newblock {\em Ann. Probab.}, 8(4):802--813, 1980.

\bibitem{jog77}
K.~Jogdeo.
\newblock Association and probability inequalities.
\newblock {\em Ann. Statist.}, 5(3):495--504, 1977.

\bibitem{marsh-olkin74}
A.~W. Marshall and I.~Olkin.
\newblock Majorization in multivariate distributions.
\newblock {\em Ann. Statist.}, 2:1189--1200, 1974.

\bibitem{marsh-ol}
A.~W. Marshall and I.~Olkin.
\newblock {\em Inequalities: theory of majorization and its applications},
  volume 143 of {\em Mathematics in Science and Engineering}.
\newblock Academic Press Inc. [Harcourt Brace Jovanovich Publishers], New York,
  1979.

\bibitem{mathew}
T.~Mathew and K.~Nordstr{\"o}m.
\newblock Inequalities for the probability content of a rotated ellipse and
  related stochastic domination results.
\newblock {\em Ann. Appl. Probab.}, 7(4):1106--1117, 1997.

\bibitem{muirRJ}
R.~F. Muirhead.
\newblock Some methods applicable to identities and inequalities of symmetric
  algebraic functions of $n$ letters.
\newblock {\em Proc. Edinburgh Math. Soc.}, 21:144--157, 1903.

\bibitem{d_to_infty}
I.~Pinelis.
\newblock Asymptotic efficiency of $p$-mean tests for means in high dimensions
  (preprint).

\bibitem{pin99}
I.~Pinelis.
\newblock Fractional sums and integrals of {$r$}-concave tails and applications
  to comparison probability inequalities.
\newblock In {\em Advances in stochastic inequalities ({A}tlanta, {GA}, 1997)},
  volume 234 of {\em Contemp. Math.}, pages 149--168. Amer. Math. Soc.,
  Providence, RI, 1999.

\bibitem{straw}
W.~E. Strawderman.
\newblock Minimax estimation of location parameters for certain spherically
  symmetric distributions.
\newblock {\em J. Multivariate Anal.}, 4:255--264, 1974.

\bibitem{tong}
Y.~L. Tong.
\newblock {\em The multivariate normal distribution}.
\newblock Springer Series in Statistics. Springer-Verlag, New York, 1990.

\end{thebibliography}

\end{document}